\documentclass{article}
\usepackage{graphicx} 
\usepackage[margin=1.0in]{geometry}
\usepackage{color,latexsym,amsmath,amssymb, amsthm, graphicx,subfigure}
\usepackage[percent]{overpic}

\newcommand{\RR}{\mathbb{R}}

\newcommand{\eps}{\varepsilon}

\newtheorem{thm}{Theorem}[section]
\newtheorem{lem}[thm]{Lemma}
\newtheorem{clm}[thm]{Claim}
\newtheorem{prop}[thm]{Proposition}

\theoremstyle{definition}

\theoremstyle{remark}

\theoremstyle{definition}
\newtheorem{exmp}[thm]{Example}

\title{Improved bound for the $k$-variate Elekes--R\'onyai theorem}
\author{Yaara Jahn and Orit E. Raz}
\author{
Yaara Jahn\thanks{Hebrew University of Jerusalem, Jerusalem, Israel. Email: yaara.jahn@mail.huji.ac.il}\and
Orit E. Raz\thanks{Ben-Gurion University of the Negev, Beer Sheva, Israel. Email: oritraz@bgu.ac.il}}
\begin{document}

\maketitle

\begin{abstract}
Let $f\in \RR[x_1,\ldots, x_k]$, for $k\ge 2$. For any finite sets $A_1,\ldots, A_k\subset \RR$, consider the set 
$$
f(A_1,\ldots, A_k):=\{f(a_1,\ldots, a_k)\mid (a_1,\cdots,a_k)\in A_1\times\cdots \times A_k\},
$$
that is, the image of $A_1\times \cdots\times A_k$ under $f$.
Extending a theorem of Elekes and R\'onyai, which deals with the case $k=2$, and the result of Raz, Sharir, and De Zeeuw~\cite{RazShaDeZee4D}, dealing with the case $k=3$, it is proved in Raz and Shem Tov~\cite{RazShe}, that for every choice of finite $A_1,\ldots, A_k\subset \RR$, each of size $n$, one has 
\begin{equation}\label{RSbound}
|f(A_1,\ldots,A_k)|=\Omega(n^{3/2}),
\end{equation}
unless $f$ has some degenerate special form.

In this paper, we introduce the notion of a \emph{rank} of a $k$-variate polynomial $f$, denoted as ${\rm rank}(f)$. 
Letting $r={\rm rank}(f)$, we prove that
\begin{equation}
|f(A_1,\ldots,A_k)|=\Omega\left(n^{\frac{5r-4}{2r}-\eps}\right),
\end{equation}
for every $\eps>0$, where the constant of proportionality depends on $\eps$ and on $\deg(f)$.
This improves the lower bound \eqref{RSbound}, for polynomials $f$ for which ${\rm rank}(f)\ge 3$.

We present an application of our main result, to lower bound the number of distinct $d$-volumes spanned by $(d+1)$-tuples of points lying on the moment curve in $\RR^d$.
\end{abstract}

\vspace{1cm}
\section{Introduction}

In many cases in combinatorial geometry, counting questions involving distances, slopes, collinearity, etc., can be reformulated as analogous counting questions involving grid points lying on certain algebraic varieties. A unified study of such problems began with a question of Elekes~\cite{Ele} about expansion of bivariate real polynomials $f(x,y)$. 
Specifically, he asked: For a bivariate polynomial $f\in \RR[x,y]$ and given finite sets $A,B\subset \RR$, how small can be the image set
$$
f(A,B)=\{f(a,b)\mid a\in A, b\in B\}.
$$

Elekes conjectured that the image of $f$ on a $n\times n$ Cartesian product must be of cardinality superlinear in $n$, unless $f$ has a very concrete {\it special form}. This was confirmed in 2000 by Elekes and R\'onyai~\cite{EleRon00} who proved the following dichotomy:
Either $f$ has one of the forms
\begin{align}
f(x,y)&=h(p(x)+q(y))\quad\text{or}\nonumber\\
f(x,y)&=h(p(x)q(y)),\label{specialbivariate}
\end{align} for some univariate real polynomials $p,q,h$, or, otherwise, for every finite $A,B\subset \RR$, each of size $n$, we have
\begin{equation}\label{eq:ER}
|f(A,B)|=\omega(n).
\end{equation}

In case $f$ is not one of the forms in \eqref{specialbivariate}, the lower bound on $|f(A,B)|$ was improved in \cite{RazShaDeZee} to be  $\Omega(n^{4/3})$, and further improved in \cite{SolZha} to be $\Omega(n^{3/2})$, which is currently the best known lower bound for bivariate polynomials that are not special.

An analogue of the Elekes--R\'onyai problem can be formulated for polynomials in more than two variables. The trivariate case was studied by Raz, Sharir, and De Zeeuw~\cite{RazShaDeZee4D}, and the general $k$-variate case was established by Raz and Shem Tov~\cite{RazShe}. They obtain the following result. 
\begin{thm}[{\bf \cite{RazShaDeZee4D, RazShe}}]\label{thm:kvarER}
    Let $k\ge 3$ and let $f\in \RR[x_1,\ldots, x_k]$. Then one of the following holds:\\
    (i) For every $A_1,\ldots, A_k\in \RR$ each of size $n$ one has
    $$
    f(A_1,\ldots,A_k)|=\Omega\left(n^{3/2}\right)$$\\
    (ii) $f$ is of one of the forms:
    \begin{align}
        f(x_1,\ldots,x_k)&=h(p_1(x_1)+\cdots +p_k(x_k))\label{kspecial}\\
        f(x_1,\ldots,x_k)&=h(p_1(x_1)\cdot \ldots \cdot p_k(x_k))\nonumber
    \end{align}
\end{thm}

Note that the bound in Theorem~\ref{thm:kvarER}, for non-special polynomials $f$, is independent of $k$, and in particular coincides with the bound for $k=3$.
At first glance this may appear to be merely a consequence of the proof. Indeed, the argument in \cite{RazShe} reduces the $k$-variate case for $k \ge 4$ to the trivariate case by fixing values for $k-3$ of the variables. They then show that if fixing any such subset of $k-3$ variables yields a special trivariate polynomial, then $f$ itself, as a $k$-variate polynomial, must be special in the sense of \eqref{kspecial}.

It is natural to expect that increasing the number of variables should force the image of $f$ to grow faster. However, certain polynomials in many variables can in effect behave like polynomials in fewer variables. For example, consider the $(k+2)$-variate polynomial
$$
f(x,y,z_1\ldots,z_k)=xy+z_1+z_2+\cdots +z_k.$$
Let $A,B,C_1,\ldots,C_k\subset \RR$, where $A,B$ are arbitrary finite sets of size $n$ and $C_1=\cdots=C_k=[n]$. Let 
$$
C:=C_1+\cdots +C_k=\{k,k+1,\ldots,kn\}.
$$ 
Then $|C|=\Theta(n)$, and letting $g(x,y,z):=xy+z$, we have
$$
f(A,B,C_1,\ldots,C_k)=g(A,B,C).$$
In this case, with the current techniques, it is unclear how to obtain a bound on the expansion of $f$ that improves upon the trivariate result for $g$.

\paragraph{Our results.}
In this paper, we recognize $k$-variate polynomials that are, in a precise sense, truly $k$-variate, and we improve the corresponding expansion bounds for them. More precisely, for a $k$-variate polynomial $f$, we introduce the notion of the {\it rank} of $f$. If $f$ has rank $r$, then, in a rigorous sense, it is essentially $(r+1)$-variate, and the bound on the size of its image can be improved with an exponent that grows with $r$.

We now define the rank of a polynomial and then state our main result.

Let $f\in \RR\left[x_1,\ldots,x_k\right]$ and let $d_{x_1}$ stand for the degree of $f$ with respect to the variable $x_1$. 
Write $$
f\left(x_1,\ldots,x_k\right)=\sum\limits_{i=0}^{d_{x_1}}\alpha_{i}\left(x_2,\ldots,x_k\right)x_1^i.
$$
We consider the {\it coefficient map} $T=T_{f,x_1}:\RR^{k-1}\to \RR^{d_{x_1}+1}$ given by $$
\left(x_2,\ldots,x_k\right)\mapsto \left(\alpha_0\left(x_2,\ldots,x_k\right),\ldots,\alpha_{d_{x_1}}\left(x_2,\ldots,x_k\right)\right).
$$
We define the rank of $f$ with respect to the variable $x_1$ to be 
$$
{\rm rank}_{x_1}(f):={\rm rank}(J_T),$$
where $J_T$ stands for the Jacobian matrix of $T$.
Note that 
$$
0\le {\rm rank}_{x_1}(f)\le k-1.
$$ 
Similarly, define ${\rm rank}_{x_i}(f)$, for every $i=2,\ldots, k$, where $x_i$ plays the role of $x_1$.

Finally, define the {\it rank} of the polynomial $f$ to be 
$$
{\rm rank}(f):=\max_{1\le i\le k}{\rm rank}_{x_i}(f).
$$

\begin{exmp} \label{example}
Let $$
f\left(x_1,x_2,\ldots,x_k\right) =  x_1x_k + x_2x_k^2\dots + x_{k-1}x_k^{k-1}.$$
Then $\text{rank}(f)=\text{rank}_{x_k}(f)=k-1$.
\end{exmp} 

\begin{exmp}
Let $$
f(x_1,x_2,\ldots,x_k) =  p_1(x_1)x_k + p_2(x_1,x_2)x_k^2+\dots + p_{k-1}(x_1,\ldots,x_{k-1})x_k^{k-1},$$
where $p_i$ is an $i$-variate polynomial that depends non-trivially on $x_i$.
Then $\text{rank}(f)=\text{rank}_{x_k}(f)=k-1$. Indeed, in this case the matrix $J_{T_{f,x_k}}$ is upper-triangular.
\end{exmp}

We prove the following main result of the paper.
\begin{thm}\label{mainthm}
Let $k\ge 3$ and let $f\in\RR[x_1\ldots,x_k]$. 
Assume that ${\rm rank}(f)=r\ge 2$.
Then, for every $\varepsilon>0$, the following holds: 
Let $A_1,\ldots,A_k\subseteq \RR$ be finite, each of size $n$. Then $$
\lvert f(A_1,\ldots,A_k)\rvert= \Omega\left( n^{\frac{5r-4}{2r}-\varepsilon}\right), 
$$
where the constant of proportionality depends on $\deg(f)$, on $r$, and on $\eps$.
\end{thm}


 We observe  that ${\rm rank}(f)=1$ corresponds to the special forms from Theorem~\ref{thm:kvarER}. Indeed, we have the following theorem. 
 \begin{thm}\label{rank1char}
Let $k\ge3$ and let $f\in\RR\left[x_1 \ldots,x_k\right]$. Assume that $f$ depends non-trivially on each of its variables and that  $$
{\rm rank}(f) = 1.$$
Then $f$ has one of the forms
\begin{align}
f\left(x_{1},\ldots,x_{k}\right) & =h\left(p_{1}\left(x_{1}\right)+\cdots+p_{k}\left(x_{k}\right)\right)\ \ \text{ or}\nonumber\\
f\left(x_{1},\ldots,x_{k}\right) & =h\left(p_{1}\left(x_{1}\right)\cdot\ldots\cdot p_{k}\left(x_{k}\right)\right),\label{special1}
\end{align}
for some univariate real polynomials $h(x), p_1(x),\ldots,p_k(x)$.
\end{thm}

 Finally, 
 we present an application of our results to the following Erd\H{o}s-type combinatorial geometric problem. 
 Let $\nu$ denote the moment curve in $\RR^d$, parameterized by
 $$\nu(t)=(t,t^2,\ldots,t^d),\quad t\in \RR.$$
 Let $P\subset \nu$ be a finite set of $n$ points. For any distinct $p_1,\ldots, p_{d+1}\in \nu$, let $\sigma=\sigma(p_1,\ldots,p_{d+1})$ denote the $d$-simplex which is the convex hull of $p_1,\ldots,p_{d+1}$ in $\RR^d$, and let  ${\rm vol}(\sigma)$ denote its $d$-dimensional  volume.
 Define
 $$
 \Delta(P)= \left\{{\rm vol}(\sigma(p_1,\ldots,p_{d+1}))\mid p_1,\ldots,p_{d+1}\in P\right\}.$$

 We have the following theorem.
\begin{thm}\label{thm:app}
    Let $\nu$ be  the moment curve in $\RR^d$ and let $P\subset \nu$ be any finite set of size $n$. Then, for every $\eps>0$, 
    $$
|\Delta(P)|=\Omega\left(n^{\frac{5d-4}{2d}-\eps}\right),
$$
    where the implicit constant depends only on $\eps$ and on $d$.
\end{thm}

 Theorem~\ref{thm:app} is obtained by identifying a $(d+1)$-variate polynomial $f$ whose expansion over a certain $n\times \cdots \times n$ grid in $\RR^{d+1}$ corresponds to the number of distinct volumes of $d$-simplices spanned by $P$. We then show that $f$ has rank $d$ and apply our main Theorem~\ref{mainthm}. 
 
\paragraph{Organization of the paper.} The paper is organized as follows. In Section~\ref{sec:pre} we recall an incidence bound that will serve as a key tool in our arguments. In Section~\ref{sec:special}, we establish a special case of our main result, Theorem~\ref{mainthm}, and in Section~\ref{sec:proofmain}, we complete its proof. The proof of 
Theorem~\ref{rank1char} is provided in Section~\ref{sec:char}. Finally, Section~\ref{sec:app} contains the proof of Theorem~\ref{thm:app}.

\section{Incidences between points and algebraic curves}\label{sec:pre}
For a finite set of points ${\cal P}\subset \RR^2$ and a finite set of planar curves ${\cal C}$, we let $I({\cal P}, {\cal C})$ denote the set of point-curve incidences; that is
$$
I({\cal P},{\cal C})=\{(p,\gamma)\in {\cal P}\times{\cal C}\mid p\in \gamma\}.
$$
The classical Szemer\'edi--Trotter theorem~\cite{SzeTro} asserts that, for the special case where ${\cal C}$ is a set of lines, and putting  $m:=|{\cal P}|$ and $n:=|{\cal C}|$, one has 
$$
\left|I({\cal P},{\cal C})\right|=O\left(m^{2/3}n^{2/3}+m+n\right).
$$

Since the Szemer\'edi--Trotter result, many alternative proofs and analogue problems have been studied. Today incidence problems play a fundamental role in combinatorial geometry.
For our result we will need an extension of the Szemer\'edi--Trotter theorem to point-curve incidence problems, where the curves are algebraic and come from an $s$-dimensional family of curves. We now present the definition from Sharir--Zahl~\cite{ShaZah}.

A bivariate polynomial $h\in\RR[x,y]$ of degree at most $D$ is a linear combination of the form
$$
h(x,y)=\sum_{0\le i+j\le D}c_{ij}x^iy^j.
$$
Note that the number of monomials $x^iy^j$ such that $0\le i+j\le D$ is $\binom{D+2}{2}$.
In this sense, every point $\vec{c}\in \RR^{\binom{D+2}{2}}$ (other than the all-zero vector) can be associated with a curve in $\RR^2$, given by the zeroset of the bivariate polynomial whose coefficients are the entries of $\vec{c}$. 
If $\lambda\neq 0$, then $f$ and $\lambda f$ have the same zero-set.
Thus, the set
of algebraic curves that can be defined by a polynomial of degree at most $D$ in $\RR^2$ can
be identified with the points in the projective space 
${\bf P}\RR^{\binom{D+2}{2}}$.

In \cite{ShaZah}, Sharir and Zahl defined an \emph{$s$-dimensional family of plane curves of degree at most $D$} to be  an algebraic variety $F\subset{\bf P}\mathbf \RR^{\binom{D+2}{2}}$ such that $\dim(F)=s$. We will call the degree of the variety $F$ the \emph{complexity of the family}. 

They then proved the following incidence bound:
\begin{thm}[{\bf Sharir--Zahl~\cite{ShaZah}}]\label{SharirThm} 
Let $\cal{F}$ be an $s$-dimensional family of plane curves of degree at most $D$ and complexity at most $K$.
Let $\mathcal P$ be a set of $m$ points in the plane and let ${\cal C}\subset{\cal F}$ be a set of $n$ plane curves. 
Suppose that no two of the curves in ${\cal C}$ share a common irreducible component.  Then
for every $\varepsilon >0$, 
we have
$$
I\left(\mathcal P,\mathcal C\right) = 
O_\varepsilon\left(m^{\frac{2s}{5s-4}}n^{\frac{5s-6}{5s-4}+\varepsilon} \right) + 
O\left(m^{2/3} n^{2/3}+m+n \right),
$$
where the constant of proportionality depends on $s,K,D$ and in the first term also on $\varepsilon$.
\end{thm}

\section {A special case}\label{sec:special}
In this section we  prove the special case of Theorem~\ref{mainthm} with $r=k-1$. This will serve as a key ingredient for the proof of the general result.

\begin{prop}\label{r_eq_k}
Let $k\ge 3$, $f\in\RR\left[x_1\ldots,x_k\right]$, and assume that ${\rm{rank}}(f)=k-1$. Then, for every $\varepsilon>0$,
the following holds: Let $A_0,\ldots,A_k\subseteq \RR$ be finite, each of size $n$. Then  
$$
\lvert f(A_1,\ldots,A_k)\rvert
= 
\Omega\left(n^{\frac{5(k-1)-4}{2(k-1)}-\varepsilon}\right), 
$$
where the constant of proportionality depends on $\deg(f)$, on $k$, and on $\varepsilon$.
\end{prop}

\begin{proof}
Let $f,A_1,\ldots,A_k$ be as in the statement. Up to renaming of the variables, we may assume without loss of generality that ${\rm rank}_{x_1}(f)=k-1$. 
Put $$
B = f(A_1,\ldots,A_k)
$$
We aim to lower bound $|B|$. 
For this, consider: $$
S = \left\{ \left(x_1,x_2,\ldots,x_k,y\right)\in A_1 \times A_2\times \cdots \times A_k\times B \mid y = f\left(x_1,\ldots,x_k\right)\right\}. 
$$
Notice that, for every $(a_1\ldots,a_k)\in  A_1\times\ldots\times A_k$, we have that 
$(a_1\ldots,a_k,f(a_1,\ldots,a_k))\in S$, and so 
\begin{equation}\label{sizeS}
\lvert S\rvert = n^{k}. 
\end{equation}

We claim that, for every $\varepsilon>0$, one has
\begin{equation}\label{upbndS}
\lvert S\rvert =O\left( \lvert B \rvert^{\frac{2(k-1)}{5(k-1)-4}}n^{(k-1)(1+\varepsilon)}\right),
\end{equation}
where the constant of proportionality depends on $\eps$, $k$, and $\deg(f)$.
Combining \eqref{sizeS} and \eqref{upbndS},  we get 
$$
n^k =O\left( \lvert B\rvert^{\frac{2(k-1)}{5(k-1)-4}}n^{(k-1)(1+\varepsilon)}
\right)$$
or
$$
|B|=\Omega 
\left(
n^{\frac{5(k-1)-4}{2(k-1)}-\eps'}
\right),$$
where $\eps'=\tfrac{5(k-1)-4}{2}\eps$, which proves the proposition. 

So in order to complete the proof of Proposition~\ref{r_eq_k}
we only need to prove \eqref{upbndS}. Let $T=T_{f,x_1}$ be the coefficient map defined in the introduction. By assumption, ${\rm rank}(J_T)=k-1$. Thus, there exist indices $\left(i_1,\ldots, i_{k-1}\right) $ such that for 
$$
\hat{T}:(x_2\ldots,x_k)
\mapsto(\alpha_{i_1}(x_2\ldots,x_k),\ldots,\alpha_{i_{k-1}}(x_2\ldots,x_k)),
$$ 
we have  
$$\det J_{\hat{T}} \not\equiv 0.$$
Define 
\begin{align*}
S_0 &=\left\{\left(a_{1},\ldots,a_{k},b\right)\in S \mid\det J_{\hat{T}}\left(a_{2},\ldots,a_{k}\right)=0\right\}, \\
S' & = S\setminus S_0.
\end{align*} 
Clearly $|S|= |S_0|+|S'|$.
Observe that,
$$
|S_0|=|A_1|\cdot\left|\left(A_2\times\cdots\times A_k\right) \cap \left\{\det J_{\hat{T}}=0\right\}\right|\le n\cdot \deg(\det J_{\hat T})n^{k-2},
$$
where the inequality is due to the Schwartz--Zippel Lemma (see \cite{Schw,Zipp}).
Thus, we get
\begin{equation}
\label{S0bound}
\left| S_0\right|=O(n^{k-1}),
\end{equation}
where the constant of proportionality depends on $\deg(f)$ and on $k$. 

We now bound $\lvert S'\rvert$. For this, we reduce the problem into a point-curve incidence problem in the plane as follows.
With each $(a_2,\ldots,a_k)\in A_2\times\cdots\times A_k$ for which $\det J_{\hat{T}}(a_2,\ldots,a_k) \neq 0$, we associate a curve $\gamma_{a_2,\ldots, a_k}$ in $\RR^2$ given by the equation $$y=f(x,a_2,\ldots,a_k).$$
Note that $\gamma_{a_2,\ldots, a_k}$ is irreducible for every $(a_2,\ldots,a_k)\in A_2\times\cdots\times A_k.
$ Let 
\begin{align*}
\mathcal P & =A_{0}\times B,\\
\mathcal C & =\left\{ \gamma_{a_{2},\ldots,a_{k}}\mid \det J_{\hat{T}}(a_1,\ldots,a_k)\neq 0\right\};
\end{align*}
note that curves in ${\cal C}$ are taken without multiplicity. 
Let $I(\mathcal P,\mathcal C)$ denote the set of point-curve incidences between $\mathcal P$ and $\mathcal C$.

\begin{clm}\label{clm:incidencereduction}
We have
\begin{equation}
\lvert S'\rvert =\Theta\left( \lvert I(\mathcal P,\mathcal C)\rvert\right),
\end{equation}
where the constant of proportionality depends only on $\deg f$ and on $k$.
\end{clm}
\begin{proof}
By definition, if $(a_1,a_{2},\ldots,a_{k},b)\in S'$ then $((a_1,b),\gamma_{a_2,\ldots,a_k})\in I(\mathcal P,\mathcal C)$.
So, to prove the claim, it suffices to show that every $(p,\gamma)\in I(\mathcal P,\mathcal C)$ corresponds to at most $O(1)$ elements of $S'$.

For $\gamma\in \mathcal C$, write 
$$
m(\gamma)=\left\{(a_2,\ldots,a_k)\mid \text{$\gamma_{a_2,\ldots,a_k}=\gamma$ and $\det J_{\hat T}(a_2,\ldots,a_k)\neq 0$}\right\}.
$$
We need to prove that 
\begin{equation}\label{mgamma}
|m(\gamma)|=O(1),
\end{equation}
with constant of proportionality that depends only on $\deg(f)$ and on $k$. 

By the definition of the set $\mathcal C$, we have that $\gamma$ is given by an equation of the form
$$
y=\sum_{i=0}^{d_{x_1}}c_ix^i,
$$
for some coefficients $c_0,\ldots, c_{d_{x_1}}\in \RR$.
Let $V$ denote the algebraic variety which is given by the system of equations
\begin{align}
\alpha_{i_1}(x_2,\ldots,x_k)&=c_{i_1}\nonumber\\
\alpha_{i_2}(x_2,\ldots,x_k)&=c_{i_2}\label{sysmgamma}\\
\vdots&\nonumber\\
\alpha_{i_{k-1}}(x_2,\ldots,x_k)&=c_{i_{k-1}}.\nonumber
\end{align}
Then 
$$
m(\gamma)\subset V.
$$

Write $V=V_0\cup V_+$, where $V_0$ is the union of all $0$-dimensional irreducible components of $V$, and $V_+$ is the union of all other irreducible components of $V$. 
Recall that, by properties of real algebraic varieties, $V_0$ is finite, and 
$$
|V_0|=O(1),
$$ 
with a constant that depends only of $\deg(f)$ and on $k$.

Thus, to prove \eqref{mgamma}, it suffices to show that $m(\gamma)\subset V_0$. 
Assume, for contradiction, that this is not the case. That is, there exists $(a_2,\ldots,a_k)\in m(\gamma)\cap V_+$. 
Then, by definition, we have 
\begin{align*}
\hat T&(a_2,\ldots,a_k)=(c_{i_1},\ldots,c_{i_{k-1}})\\
\det &J_{\hat T}(a_2,\ldots,a_k)\neq 0.
\end{align*}
By the inverse function theorem, there exists an open neighborhood, $N$, of $(a_2,\ldots,a_k)$ such that $\hat T$ restricted to $N$ is invertible. 
In particular,
$$
N\cap \hat T^{-1}\{(c_{i_1},\ldots,c_{i_{k-1}})\}= \{(a_2,\ldots,a_k)\}.
$$
On the other hand,
$$
N\cap \hat T^{-1}(c_{i_1},\ldots,c_{i_{k-1}})= N\cap V.
$$
Since $N$ is a neighborhood of $(a_1,\ldots,a_{k-1})$ and the latter lies on an irreducible component of $V$ of dimension at least 1, the intersection $N\cap V$ must be infinite. This leads to a contradiction, and thus $m(\gamma)\subset V_0$. This completes the proof of Claim~\ref{clm:incidencereduction}.
\end{proof}

In view of Claim~\ref{clm:incidencereduction}, in order to bound $|S'|$ it suffices to bound $\left|I\left(\mathcal P,\mathcal C\right)\right|.$ We have 
$\lvert\mathcal P\rvert = n\lvert B\rvert $ and $\lvert\mathcal C\rvert \le n^k$. 
Since the curves in $\mathcal C$ are irreducible, every two distinct curves $\gamma,\gamma'$ in $\mathcal C$ intersect in at most $d_{x_1}^2$ points, by Bezout's Theorem (see e.g. \cite[Corollary 7.8]{Hart}). 
We can therefore apply 
Theorem~\ref{SharirThm}, which gives
\begin{align*}
\lvert I\left(\mathcal P,\mathcal C\right)\rvert  
&=O\left( \left(\lvert B\rvert\cdot n\right)^{\frac{2(k-1)}{5(k-1)-4}}n^{(k-1)\left(\frac{5(k-1)-6}{5(k-1)-4}
+\varepsilon\right)}
+(\lvert B\rvert n)^{2/3}(n^{k-1})^{2/3}+\lvert B\rvert n+n^{k-1}\right).
\end{align*}
or
\begin{align*}
\lvert I\left(\mathcal P,\mathcal C\right)\rvert  
&=O\left( \lvert B \rvert^{\frac{2(k-1)}{5(k-1)-4}}n^{(k-1)\left(1+\varepsilon\right)}+|B|^{\frac23}n^{\frac{2k}{3}}+
\lvert B\rvert n+n^{k-1}\right).
\end{align*}
Note that we may assume without loss of generality that  the first summand is dominant. Indeed, 
the second summand is dominant if
$$|B|^{\frac{2(k-1)}{5(k-1)-4}}n^{(k-1)(1+\varepsilon)}\le|B|^{\frac{2}{3}}n^{\frac{2k}{3}}$$
or 
$$
|B|\ge 
n^{\frac{5(k-1)-4}{4}+O(\varepsilon)},
$$
which is stronger than the lower bound we wish to prove on $|B|$, for every $k\ge 3$.
Similarly, the third summand is dominant 
if 
$$|B|^{\frac{2r}{5r-4}}n^{r(1+\varepsilon)}\le |B|n
$$
or 
$$|B|\ge 
n^{\frac{(k-2)(5(k-1)-4)}{3(k-1)-4}+\eps\frac{(k-1)(5(k-1)-4)}{3(k-1)-4} }
,
$$
which is better than the lower we want to prove on $|B|$ for every $k$, as is easy to verify. Finally, the fourth summand is always subsumed by the first one.

Hence, either the conclusion of Proposition~\ref{r_eq_k} holds, or we obtain
$$
\lvert I\left(\mathcal P,\mathcal C\right)\rvert  
=
O\left(
\lvert B \rvert^{\frac{2(k-1)}{5(k-1)-4}}n^{(k-1)\left(1+\varepsilon\right)}
\right).
$$
In view of Claim~\ref{clm:incidencereduction} and combined with \eqref{S0bound}, the inequality \eqref{upbndS} follows. This completes the proof Proposition~\ref{r_eq_k}.
\end{proof}

\section{Proof of Theorem~\ref{mainthm}}\label{sec:proofmain}
The following lemma shows that the main Theorem~\ref{mainthm} can in fact be reduced to the statement of Proposition~\ref{r_eq_k}.
\begin{lem}\label{lem:reductionr}
Let $f\in\RR\left[x_1\ldots,x_k\right]$ and assume that ${\rm{rank}}_{x_1}(f)=r<k-1$.
Then, up to renaming of the variables $x_2,\ldots,x_k$, we have that $$
g(x_1,\ldots,x_{r+1}):= f\left(x_1,x_2,\ldots,x_{r+1},y_{r+2},\ldots,y_k\right) $$
is a $(r+1)$-variate polynomial in $\RR\left[x_1,\ldots,x_{r+1}\right]$ with ${\rm rank}_{x_1}(g)=r$,
for all $\left(y_{r+2},\ldots,y_k\right) \in \RR^{k-r-1}\setminus Z_0$, where $Z_0$ is some subvariety of $\RR^{k-r-1}$ of codimension at least 1.
\end{lem}

\begin{proof}
Let $f$ be as in the statement. Let $T=T_{f,x_1}$ be the corresponding coefficient map. By assumption
\begin{equation}\label{rankJT}
{\rm rank}\left(J_T\right) = r < k-1.
\end{equation}
Up to renaming the variables, we may assume without the loss of generality, that the first $r$ columns of $J_T$, corresponding to the variables $x_2,\ldots,x_{r+1}$, are independent. 
Observe that the matrix composed of the first $r$ columns of $J_{T}$, is in fact the Jacobian matrix of the coefficient map $T_{g,x_1}$, where $g$ is the $(r+1)$-variate polynomial given by 
$$
g\left(x_1,x_2,\ldots,x_{r+1}\right):= f\left(x_1,x_2,\ldots,x_{r+1},y_{r+2},\ldots,y_k\right);
$$
here $y_{r+2},\ldots, y_k$ are regarded as constant parameters. 
More concretely, we have that ${\rm rank}(J_{T_{g,x_1}})=r$, for every generic 
$(x_2,\ldots , x_{r+1},y_{r+2},\ldots ,y_k)\in \RR^{k}$. 

Let $\Delta$ denote the polynomial corresponding to the sum of squares of the determinants of all the $r\times r$ submatrices of  $J_{T_{g,x_1}}$.  So $\Delta$ is a multivariate polynomial, and we can write
$$
\Delta(x_2,\ldots,x_{r+1},y_{r+2},\ldots,y_k)= \sum\limits_{i} \beta_i(y_{r+2},\ldots,y_k) g_i(x_2,\ldots,x_{r+1}),
$$
with $g_i\in\RR[x_2,\ldots,x_{r+1}]$ and $\beta_i\in\RR[y_{r+2},\ldots,y_k]$.
By the definition of $\Delta$ we have 
$$
\text{${\rm rank}(J_{T_{g,x_1}}(x_2,\ldots , x_{r+1},y_{r+2},\ldots ,y_k))<r$ if and only if $\Delta(x_2,\ldots , x_{r+1},y_{r+2},\ldots ,y_k)=0$. }
$$
Thus, in view of \eqref{rankJT}, we have  $\Delta\not\equiv 0$.
In particular, 
the polynomials $\beta_i$ are not all zero. Thus, letting $$
Z_0 := \left\{\left(y_{r+2},\ldots,y_k\right) \mid \forall i~\beta_i(y_{r+2},\ldots,y_k) = 0\right\},
$$
we see that $Z_0$  has codimension at least 1. 
This completes the proof of the lemma. 
 \end{proof}

We can now complete the proof of our main Theorem~\ref{mainthm}.
\begin{proof}[Proof of Theorem~\ref{mainthm}]
Let $f\in \RR[x_1,\ldots,x_k]$ and assume without loss of generality that $\deg_{x_1}(f)=r\le k-2$. Apply Lemma~\ref{lem:reductionr} to $f$. Then, up to renaming of the variables, there exists an algebraic variety $Z_0\subset \RR^{k-r-1}$, of codimension at least one, and of degree $O(1)$, such that for every $(y_{r+2},\ldots,y_k)\in \RR^{k-r-1}\setminus Z_0$ we have that the polynomial 
$$
(x_1,\ldots,x_{r+1})\mapsto f(x_1,\ldots,x_{r+1},y_{r+2},\ldots,y_k)
$$
is an $(r+1)$-variate polynomial of rank $r$. 

Observe that, by the Schwartz--Zippel lemma, there exists $(a_{r+2},\ldots,a_k)\in \left(A_{r+2}\times\ldots\times A_k\right)\setminus Z_0$. Thus 
$$
g(x_1,\ldots,x_{r+1}):=f(x_1,\ldots,x_{r+1},a_{r+2},\ldots,a_k)
$$
satisfies ${\rm rank}_{x_1}(g)=r$. Thus,  by Proposition~\ref{r_eq_k}, we have that 
$$
|g(A_1,\ldots,A_{r+1})|=\Omega\left(n^{\frac{5r-4}{2r}-\varepsilon}\right).
$$ 
Noting that 
$$
g(A_1,\ldots,A_{r+1})=f(A_1,\ldots,A_{r+1},\{a_{r+2}\},\ldots,\{a_k\})\subset f(A_1,\ldots,A_k),
$$
this completes the proof of the theorem. 
\end{proof}

\section{Characterization of rank-1 polynomials}\label{sec:char}

In this section we prove Theorem~\ref{rank1char}. For the proof we will use the following lemma from Raz and Shem Tov~\cite{RazShe}.
\begin{lem}[{\bf Raz--Shem Tov~\cite[Lemma 2.3]{RazShe}}]\label{RazShem}
Let $f\in\RR[x_1,\ldots,x_k]$. Assume
that
\begin{equation}\label{rank1diff}
\frac{\frac{\partial f}{\partial x_1}(x_1,\ldots,x_k)
}
{
r_1(x_1)
}
=\cdots =
\frac{\frac{\partial f}{\partial x_k}(x_1,\ldots,x_k)
}
{
r_k(x_k)
},
\end{equation}
for some univariate real polynomials $r_1,\ldots, r_k$. 
Then, $f$ is one of the forms
\begin{align*}
f\left(x_{1},\ldots,x_{k}\right) & =h\left(p_{1}\left(x_{1}\right)+\cdots+p_{k}\left(x_{k}\right)\right)\ \ \text{ or}\\
f\left(x_{1},\ldots,x_{k}\right) & =h\left(p_{1}\left(x_{1}\right)\cdot\cdots\cdot p_{k}\left(x_{k}\right)\right),
\end{align*}
for some univariate real polynomials $h(x), p_1(x),\ldots,p_k(x)$.

\end{lem}
\begin{proof}[Proof of Theorem~\ref{rank1char}]
Let $f$ be as in the statement.
Since $f$ depends non-trivially on each of its variables, we have in particular that 
\begin{equation}\label{eq:allrank1}
{\rm rank}_{x_i}(f)=1,\quad\text{ for each $i=1,\ldots,k$.}
\end{equation}

In view of Lemma~\ref{RazShem}, it suffices to show that $f$ satisfies the differential equation \eqref{rank1diff}, for some univariate polynomials $r_1,\ldots,r_k$. By symmetry, it suffices to prove that 
\begin{equation}\label{identityforx1x2}
\frac{\frac{\partial f}{\partial x_1}(x_1,\ldots,x_k)
}
{
r_1(x_1)
}
=\frac{\frac{\partial f}{\partial x_2}(x_1,\ldots,x_k)
}
{
r_2(x_2)
}
\end{equation}

We write 
\begin{equation}
f\left(x_{1},\ldots,x_{k}\right)
=\sum_{i=0}^{d_{x_{1}}}\alpha_{i}^{(1)}\left(x_{2},\ldots,x_{k}\right)x_{1}^{i}=
\cdots
=\sum_{i=0}^{d_{x_{k}}}\alpha_{i}^{(k)}\left(x_{1},\ldots,x_{k-1}\right)x_{k}^{i},
\end{equation}
where $d_{x_j}$ stands for the degree of $f$ as a univariate polynomial in the variable $x_j$. 
Then for every $j,\ell\in\left\{ 1,\ldots,k\right\}$, such that $j\neq \ell$, we have
\begin{equation}
\label{derivwithcolumns}
\frac{\partial f}{\partial x_{j}}\left(x_{1},\ldots,x_{k}\right)=\sum_{i=0}^{d_{x_{\ell}}}\frac{\partial\alpha_{i}^{(\ell)}\left(x_{1},\ldots,\widehat{x_{\ell}},\ldots,x_{k}\right)}{\partial x_{j}}x_{\ell}^{i}.
 \end{equation}

We next show that \begin{equation}\label{x1x2}
\frac{\frac{\partial f}{\partial x_{1}}}{{\frac{\partial f}{\partial x_{2}}}}=u(x_1,x_2),
\end{equation}
where $u$ is some rational function over $\RR$.

In other words, we need to show that 
$\frac{\partial f}{\partial x_{1}}/{\frac{\partial f}{\partial x_{2}}}$ is independent of $x_i$, for every $i\neq 1,2$. 
By symmetry it suffices to show that this ratio is independent of the variable $x_3$.
Consider the first two columns of $J_{T_{f,x_3}}$, corresponding to derivatives with respect to $x_1$ and to $x_2$. Namely,
$$
\left(\begin{array}{cc}
\frac{\partial\alpha_{0}^{(3)}}{\partial x_{1}} & \frac{\partial\alpha_{0}^{(3)}}{\partial x_{2}}\\
\frac{\partial\alpha_{1}^{(3)}}{\partial x_{1}} & \frac{\partial\alpha_{1}^{(3)}}{\partial x_{2}}\\
\vdots& \vdots\\
\frac{\partial\alpha_{d_{x_{3}}}^{(3)}}{\partial x_{1}} & \frac{\partial\alpha_{d_{x_{3}}}^{(3)}}{\partial x_{2}}
\end{array}\right).
$$
Note that neither of the columns is zero, since $f$ depends non-trivially in each of its variables. Moreover, since ${\rm rank}(J_{T_{f,x_3}})=1$, by \eqref{eq:allrank1}, and using the fact that the entries of $J_{T_{f,x_3}}$ are independent of $x_3$, we get that there exists a rational function $u(x_1,x_2,x_4,\ldots,x_k)$ such that
$$
\frac{\partial \alpha_i^{(3)}}{\partial x_1}
=
u(x_1,x_2,x_4,\ldots,x_k)
\frac{\partial \alpha_i^{(3)}}{\partial x_2},\quad\text{for $i=0,\ldots,d_{x_3}$.}
$$
Using the identity \eqref{derivwithcolumns}, this implies that $$\frac{\frac{\partial f}{\partial x_1}}{\frac{\partial f}{\partial x_2}}=u(x_1,x_2,x_4,\ldots,x_k)
;$$
that is, this ratio is independent of $x_3$.
This proves \eqref{x1x2}.

Repeating the analysis symmetrically for the pair $x_1,x_3$ and then to the pair $x_2,x_3$ we conclude that there exist real rational functions $v(x_1,x_3)$ and $w(x_2,x_3)$ such that 
\begin{align*}
\frac{\frac{\partial f}{\partial x_1}}{\frac{\partial f}{\partial x_3}}&=v(x_1,x_3)\\
\frac{\frac{\partial f}{\partial x_2}}{\frac{\partial f}{\partial x_3}}&=w(x_2,x_3).
\end{align*}
But then
$$
\frac{\frac{\partial f}{\partial x_1}}{\frac{\partial f}{\partial x_2}}=\frac{v(x_1,x_3)}{w(x_2,x_3)}=u(x_1,x_2), 
$$
meaning in particular that the ratio of $v,w$ is independent of $x_3$. Thus, setting an arbitrary value for $x_3$, we get that
$$
\frac{\frac{\partial f}{\partial x_1}}{\frac{\partial f}{\partial x_2}}=\frac{v(x_1,0)}{w(x_2,0)}.
$$
Thus, letting $r_1(x_1):=v(x_1,0)$ and $r_2(x_2):=w(x_2,0)$, this proves \eqref{identityforx1x2}, and hence completes the proof of the lemma.
\end{proof}

\section{
Distinct  $d$-volumes on the moment curve in $\RR^d$}\label{sec:app}
In this section we prove Theorem~\ref{thm:app}.

\begin{proof}[Proof of Theorem~\ref{thm:app}]
The {\it moment curve} in $\RR^d$ is defined by $\nu(x)=(x,x^2,\ldots,x^d)$. Consider $d+1$ point on $\nu$ given by the parameters $x_1,\ldots, x_{d+1}\in \RR^d$, and let 
$\sigma(x_1,\ldots,x_{d+1})$ denote the simplex spanned by these points. Define 
$$
f(x_1,\ldots,x_{d+1}):={\rm vol}\sigma(x_1,\ldots,x_{d+1}).
$$
It suffices to prove that
\begin{equation}\label{rankfvolume}{\rm rank}_{x_{d+1}}(f)=d.
\end{equation}
Indeed, assume that \eqref{rankfvolume} is true and let $P\subset \nu$ be a set of size $n$. Let $A$ denote the $x_1$-coordinate of the points in $P$. Then 
$\Delta(P)=f(A,\ldots,A)$. By Theorem~\ref{mainthm}, for every $\eps>0$, we have $$|\Delta(P)|=|f(A,\ldots,A)|=\Omega\left(n^{\frac{5d-4}{2d}-\eps}\right),$$ as needed. 

So we only need to prove \eqref{rankfvolume}. We have
\begin{align*}
        f(x_1,\ldots,x_{d+1})
        &=\frac{1}{d!}\det \begin{pmatrix} 
1&x_1 & x_1^2&\ldots& x_1^d\\ 
1&x_2 & x_2^2&\ldots& x_2^d\\
\vdots&\vdots&\vdots&&\vdots
\\
1&x_{d+1} & x_{d+1}^2&\ldots& x_{d+1}^d\\ 
\end{pmatrix},
\end{align*}
which is the determinant of the $(d+1)\times (d+1)$-{\it Vandermonde matrix}.
Thus \begin{align*}
f(x_1,\ldots,x_{d+1})
&=
\frac{1}{d!}
\prod_{1\le i<j\le d+1}(x_j-x_i)\\
&=
\frac{1}{d!}
\prod_{1\le i<j\le d}(x_j-x_i)\prod_{k=1}^d(x_{d+1}-x_k)\\
&=g(x_1,\ldots,x_d)\hat{f}(x_1,\ldots,x_{d+1}),
    \end{align*}
where
$
g(x_1,\ldots,x_d):=\frac{1}{d!}
\prod_{1\le i<j\le d}(x_j-x_i)
$
is independent of $x_{d+1}$ and
$
\hat f(x_1,\ldots,x_{d+1}):=\prod_{k=1}^d(x_{d+1}-x_k).$
Note that $\hat f$ can be written as
$$
\hat f(x_1,\ldots,x_{d+1})
=
\sum_{\ell=0}^d s_\ell(x_1,\ldots,x_d)x_{d+1}^\ell,
$$
where 
\begin{align*}
s_0(x_1,\ldots,x_d)
&=
(-1)^dx_1\cdots x_d\\
s_1(x_1,\ldots,x_d)
&=
(-1)^{d-1}(x_2\cdots x_d+x_1x_3\cdots x_d+\cdots +x_1\cdots x_{d-1})\\
&\ldots\\
s_{d-1}(x_1,\ldots,x_d)
&=-x_1-\cdots-x_d\\
s_d(x_1,\ldots,x_d)
&=1,
\end{align*}
are the symmetric polynomials. 
\begin{clm}\label{symminvertible}
Let 
$M$ be the $d\times d$ matrix 
$M=(M_{i,j})_{0\le i\le d-1,\;1\le j\le d}$ given by
$$ M_{i,j} = \frac{\partial s_i}{\partial x_j}\qquad(0\le i\le d-1,\;1\le j\le d). $$
Then ${\rm rank}(M)=d$. 
\end{clm}
The strategy for this proof is adapted from the proof of~\cite[Section 4]{Que}. 
\begin{proof} 
We prove that
$\det M$ is not the zero polynomial. 
Consider the function
$$ P(x_1,\ldots,x_d,t)=\prod_{k=1}^d (t-x_k)
= \sum_{k=0}^d s_k\, t^k,
$$
which is a monic polynomial in $t$.
For $1\le j\le d$ we have
$$
\frac{\partial P}{\partial x_j}(t)
=-\prod_{\substack{k=1\\ k\ne j}}^d (t-x_k)
=\sum_{k=0}^{d-1} \frac{\partial s_k}{\partial x_j}\,t^k
$$

Evaluating at $t=x_i$ yields
\begin{equation}
\label{identity}
\frac{\partial P}{\partial x_j}(x_i)
=\sum_{k=0}^{d-1}\frac{\partial s_k}{\partial x_j}\,x_i^k
 =
\begin{cases}
0,& i\ne j,\\[4pt]
-\displaystyle\prod_{k\ne j}(x_j-x_k),& i=j.
\end{cases} \end{equation}
The identity \eqref{identity} implies
$$
\begin{pmatrix} 
1&x_1 & x_1^2&\ldots& x_1^{d-1}\\ 
1&x_2 & x_2^2&\ldots& x_2^{d-1}\\
\vdots&\vdots&\vdots&&\vdots
\\
1&x_{d} & x_{d}^2&\ldots& x_{d}^{d-1}\\ 
\end{pmatrix}M=
\operatorname{diag}\left(\frac{\partial P}{\partial x_1}(x_1),\ldots,\frac{\partial P}{\partial x_d}(x_d)\right).
$$
Thus
\begin{align*}
\det(M)
&=
\frac
{
(-1)^d\prod_{j=1}^d\prod_{i\neq j}(x_j-x_i)
}
{
\prod_{1\le i<j\le d}(x_j-x_i)
}\\
&=
\frac
{
(-1)^d\prod_{1\le i\neq j\le d}(x_j-x_i)
}
{
\prod_{1\le i<j\le d}(x_j-x_i)
}\\
&=
\prod_{1\le i<j\le d}(x_j-x_i).
\end{align*}
This completes the proof of Claim~\ref{symminvertible}.
 \end{proof}

We are now ready to prove \eqref{rankfvolume}.
Write
$$
f(x_1,\ldots,x_{d+1})=
\sum_{i=0}^d\alpha_i(x_1,\ldots,x_d)x_{d+1}^i.$$
By the above we have
$$
\alpha_i=g(x_1,\ldots,x_d)s_{i}(x_1,\ldots,x_d),\text{ for $i=0,\ldots,d$}
$$
Let $T=T_{f,x_{d+1}}$. We have
$$
J_{T}= \begin{pmatrix}
g_{x_1}s_0+g{s_0}_{x_1} & g_{x_2}s_0+g{s_0}_{x_2} & 
\cdots
&
g_{x_d}s_0+g{s_0}_{x_d}\\
\vdots&\vdots&&\vdots\\
g_{x_1}s_{d-1}+g{s_{d-1}}_{x_1} & g_{x_2}s_{d-1}+g{s_{d-1}}_{x_2} & 
\cdots
&
g_{x_{d}}s_{d-1}+g{s_{d-1}}_{x_{d-1}}\\
g_{x_1}&g_{x_2}&\cdots&g_{x_d}
\end{pmatrix}.
$$
Applying row operation to the first $d$ rows, namely, $R_i\leftarrow R_i-s_iR_d$ we get
\begin{align*}
{\rm rank}J_{T}
&= {\rm rank}\begin{pmatrix}
g{s_0}_{x_1} & g{s_0}_{x_2} & 
\cdots
&
g{s_0}_{x_d}\\
\vdots&\vdots&&\vdots\\
g{s_{d-1}}_{x_1} & g{s_{d-1}}_{x_2} & 
\cdots
&
g{s_{d-1}}_{x_{d-1}}\\
g_{x_1}&g_{x_2}&\cdots&g_{x_d}
\end{pmatrix}\\
&={\rm rank} \begin{pmatrix}
{s_0}_{x_1} & {s_0}_{x_2} & 
\cdots
&
{s_0}_{x_d}\\
\vdots&\vdots&&\vdots\\
{s_{d-1}}_{x_1} & {s_{d-1}}_{x_2} & 
\cdots
&
{s_{d-1}}_{x_{d-1}}\\
g_{x_1}/g&g_{x_2}/g&\cdots&g_{x_d}/g
\end{pmatrix}.
\end{align*}
Note that first $d$ rows of the above matrix are exactly the matrix $M$ from Claim~\ref{symminvertible}. This completes the proof of Theorem~\ref{thm:app}
\end{proof}

\end{document}